\documentclass[11pt]{article}
\usepackage{a4wide,amsfonts,amssymb,amsmath,amsthm, amssymb}
\usepackage{pgf,pgfarrows,pgfnodes,pgfheaps}

\definecolor{softred}{gray}{0.8}
\definecolor{softgreen}{gray}{0.7}
\definecolor{softblue}{gray}{0.6}

\definecolor{softrb}{gray}{0.7}

\def\proof{{\bf Proof\quad}}
\def\beginpf{\proof}
\def\qed{\hfill\rule{2.2mm}{2.2mm}\vspace{1ex}}
\def\endpf{\qed}

\newtheorem{theorem}{Theorem}[section]

\newtheorem{definition}[theorem]{Definition}

\newtheorem{proposition}[theorem]{Proposition}
\newtheorem{remark}[theorem]{Remark}

\def\epsilon{\varepsilon}

\def\YY{\mathcal Y}
\def\CC{\mathbb C}

\def\DD{\mathbb D}

\def\NN{\mathbb N}
\def\UU{\mathcal U}

\def\TT{\mathbb T}
\def\LL{\mathcal L}

\def\cD{\mathcal D}

\def\eins{\mathbf 1}

\newcommand{\HH}{\mathcal H}
\newcommand{\KK}{\mathcal K}

\newcommand{\re}{\mathop{\rm Re}\nolimits}

\def\text{\mbox}

\def\beq{\begin{equation}}
\def\eeq{\end{equation}}

\title{$\beta$-admissibility of observation and control
operators for hypercontractive semigroups}
\author{Birgit Jacob\thanks{Fachbereich C - Mathematik und
Naturwissenschaften,
Bergische Universit\"at Wuppertal,
Gau\ss stra\ss e 20, 42119 Wuppertal,
Germany. \tt jacob@math.uni-wuppertal.de}, Jonathan R. Partington\thanks{School of Mathematics, University of Leeds, Leeds LS2 9JT, UK.
\tt j.r.partington@leeds.ac.uk},
 Sandra Pott\thanks{Faculty of Science,
Centre for Mathematical Sciences,
Lund University,
22100 Lund, Sweden. \tt sandra@maths.lth.se},
\ and Andrew Wynn\thanks{Department of Aeronautics,
Imperial College London,
London, SW7 2AZ, UK.
\tt a.wynn@imperial.ac.uk}
}

\begin{document}

\maketitle

\begin{abstract} We prove a Weiss conjecture on $\beta$-admissibility of 
control and observation operators for
discrete and continuous $\gamma$-hypercontractive semigroups of operators, by representing them
in terms of shifts on weighted Bergman spaces and using a reproducing kernel thesis for Hankel operators. Particular attention is paid to the case $\gamma=2$, which corresponds to the unweighted Bergman shift.
\end{abstract}


{\bf 2010 Subject Classification:} 30H10, 30H20, 47B32, 47B35, 47D06, 93B28

%
%
%
%
%
%
%


\section{Introduction}

We study infinite dimensional observation systems of the form
\begin{eqnarray*}
\dot{x}(t) &=& Ax(t), \quad y(t)=Cx(t),\quad t\ge 0,\\
x(0) &=& x_0\in X,
\end{eqnarray*}
where $A$ is the generator of a strongly continuous semigroup $(T(t))_{t\ge0}$ on a Hilbert space $\HH$ and $C$ is a linear bounded operator from $D(A)$, the domain of $A$ equipped with the graph topology, to another Hilbert space $\YY$. For the well-posedness of the system with respect to the output space  $L^2_\beta (0,\infty; \YY):= \{ f: (0,\infty)\rightarrow \YY \mid f \mbox{ measurable, }\|f\|_\beta^2:= \int_0^\infty \|f(t)\|^2 t^\beta \, dt<\infty \}$ it is required that $C$ is an {\em $\beta$-admissible observation operator} for $A$, that is, there exists an $M>0$ such that
\[ \|CT(\cdot)x_0\|_{L_{\beta}(0,\infty;\YY)} \le M \|x_0\|_{\HH}, \qquad x_0\in D(A).\]
It is easy to show that $\beta$-admissibility implies the resolvent condition
\begin{equation}\label{eqnresolvent}
    \sup_{\lambda\in \mathbb C_+} (\mbox{Re}\,\lambda)^{\frac{1+\beta}{2}} \|C(\lambda-A)^{-(1+\beta)}\|<\infty
\end{equation}
where $\mathbb C_+$ denotes the open right half plane of $\mathbb C$. Whether or not the converse implication holds is commonly referred to as a
{\em weighted Weiss conjecture}.  
For $\beta=0$ the conjecture was posed by Weiss \cite{we91}. In this situation the conjecture is true for contraction semigroups if the output space is finite-dimensional, for right-invertible semigroup and  for bounded analytic semigroups if $(-A)^{1/2}$ is $0$-admissible. However, in general the conjecture is not true. We illustrate this in Figure \ref{fig1}.

 \begin{figure}[h]
 \centering
 \begin{pgfpicture}{0cm}{-2cm}{10cm}{5.2cm}
   
        \pgfsetlinewidth{1pt}
        \pgfxyline(12,-2)(12,5)
        \pgfxyline(-2,5)(12,5)
         \pgfxyline(5,-2)(5,5)
          \pgfxyline(-2,-2)(12,-2)
        \pgfxyline(-2,-2)(-2,5)
      \pgfsetlinewidth{0.8pt}
            \pgfxyline(-2,4)(12,4)
    \pgfputat{\pgfxy(1.5,4.3)}{\pgfbox[center,base]{dim $\YY<\infty$}}
     \pgfputat{\pgfxy(8.5,4.3)}{\pgfbox[center,base]{dim $\YY \le \infty$}}
      \color{softgreen}
        \pgfrect[fill]{\pgfxy(-1.8,-1.2)}{\pgfxy(6.6,4.9)}
        \color{black}
       \pgfputat{\pgfxy(1.4,1.8)}{\pgfbox[center,base]{ $(T(t))_{t\ge0}$ contraction semigroup  \cite{jp01}}}
     \color{softblue}
          \pgfrect[fill]{\pgfxy(6,-1.2)}{\pgfxy(5,3)}
        \color{black}
       \pgfputat{\pgfxy(8.4,0.1)}{\pgfbox[center,base]{$(T(t))_{t\ge0}$ right-invertible}}
        \pgfputat{\pgfxy(8.5,-0.4)}{\pgfbox[center,base]{semigroup \cite{we91}}}
      \color{softred}
          \pgfrect[fill]{\pgfxy(5.1,0.7)}{\pgfxy(6.8,3)}
        \color{black}
       \pgfputat{\pgfxy(8.4,2.7)}{\pgfbox[center,base]{$(T(t))_{t\ge0}$ analytic \& bounded semigr.}}
    \pgfputat{\pgfxy(8,2.1)}{\pgfbox[center,base]{and $(-A)^{1/2}$ $0$-admissible \cite{lm03}}}    %
            \color{softrb}
      \pgfrect[fill]{\pgfxy(6,0.7)}{\pgfxy(5,0.8)}
  \color{black}
          \pgfcircle[fill]{\pgfxy(5.7,-1.6)}{0.1cm}
         \pgfputat{\pgfxy(8.7,-1.7)}{\pgfbox[center,base]{Counterexample in general \cite{japapo02}}}
 \color{black}
           \pgfcircle[fill]{\pgfxy(-1,-1.6)}{0.1cm}
         \pgfputat{\pgfxy(2,-1.7)}{\pgfbox[center,base]{Counterexample in general \cite{jazw04}}}   
    \end{pgfpicture}
    \caption{Weighted Weiss conjecture: Case $\beta=0$}
    \label{fig1}
\end{figure}

For $\beta\not=0$, there is much less known. In the situation $\beta<0$, the weighted Weiss conjecture is true for bounded analytic semigroups if $(-A)^{1/2}$ is $0$-admissible \cite{hlm}, but in general the weighted Weiss conjecture does not hold \cite{wynn09}.
If $\beta>0$, then the weighted Weiss conjecture is true for normal contraction semigroups and for the right-shift on  $L^2_{-\alpha}(0,\infty)$
for $\alpha>0$ if the the output space is finite-dimensional, and for bounded analytic semgroups if $(-A)^{1/2}$ is $0$-admissible, see Figure \ref{fig2}. Again, in general the conjecture is not true. In Theorem  \ref{thm:contadm} we show that the weighted Weiss conjecture holds if the dual of the cogenerator $T^*$ of the semigroup  $(T(t))_{t\ge0}$ is $\gamma$-hypercontractive for some $\gamma>1$. The proof is based on the fact that $\gamma$-hypercontractions are unitarily equivalent to the restriction of the backward shift to an invariant subspace of a weighted Bergman space, the Cayley transform between discrete-time and continuous-time systems and that the weighted Weiss conjecture holds for  the backward shift to an invariant subspace of a weighted Bergman space  \cite{jrw}. In order to apply the results of  \cite{jrw} we first have to extend them to the vector-valued Bergman spaces.

\begin{figure}[h]
 \centering
 \begin{pgfpicture}{0cm}{-2cm}{10cm}{5.2cm}
   
        \pgfsetlinewidth{1pt}
        \pgfxyline(12,-2)(12,5)
        \pgfxyline(-2,5)(12,5)
         \pgfxyline(5,-2)(5,5)
          \pgfxyline(-2,-2)(12,-2)
        \pgfxyline(-2,-2)(-2,5)
      \pgfsetlinewidth{0.8pt}
            \pgfxyline(-2,4)(12,4)
    \pgfputat{\pgfxy(1.5,4.3)}{\pgfbox[center,base]{dim $\YY<\infty$}}
     \pgfputat{\pgfxy(8.5,4.3)}{\pgfbox[center,base]{dim $\YY \le \infty$}}
              \color{softred}
        \pgfrect[fill]{\pgfxy(-1.8,-1.2)}{\pgfxy(6.6,4.9)}
     \color{black}
     \pgfputat{\pgfxy(1.4,1.1)}{\pgfbox[center,base]{$T^*$ $\gamma$-hypercontractive, }}
      \pgfputat{\pgfxy(1.5,0.6)}{\pgfbox[center,base]{ $\gamma>1$ (Thm. \ref{thm:contadm}) }}
   
      \color{softgreen}
        \pgfrect[fill]{\pgfxy(-1.4,1.8)}{\pgfxy(5.7,1.6)}
        \color{black}
       \pgfputat{\pgfxy(1.4,2.8)}{\pgfbox[center,base]{ $(T(t))_{t\ge0}$ normal}}
           \pgfputat{\pgfxy(1.5,2.2)}{\pgfbox[center,base]{ contraction semigroup  \cite{wynn10}}}
             \color{softgreen}
        \pgfrect[fill]{\pgfxy(-1,-1)}{\pgfxy(4.9,1.3)}
        \color{black}
       \pgfputat{\pgfxy(1.4,-0.2)}{\pgfbox[center,base]{ $(T(t))_{t\ge0}$ right-shift}}
           \pgfputat{\pgfxy(1.5,-0.8)}{\pgfbox[center,base]{ on $L^2_{-\alpha}(0,\infty)$, $\alpha>0$,  \cite{jrw}}}

     \color{softblue}
          \pgfrect[fill]{\pgfxy(5.7,-1.2)}{\pgfxy(5.5,3)}
        \color{black}
       \pgfputat{\pgfxy(8.4,0.1)}{\pgfbox[center,base]{$T^*$ $\gamma$-hypercontractive,}}
        \pgfputat{\pgfxy(8.5,-0.4)}{\pgfbox[center,base]{$\gamma>1$ (Thm. \ref{thm:contadm})}}
      \color{softred}
          \pgfrect[fill]{\pgfxy(5.2,1.1)}{\pgfxy(6.6,2.6)}
        \color{black}
       \pgfputat{\pgfxy(8.4,2.9)}{\pgfbox[center,base]{$(T(t))_{t\ge0}$ analytic \& bounded semigr.}}
    \pgfputat{\pgfxy(8,2.3)}{\pgfbox[center,base]{and $(-A)^{1/2}$ $0$-admissible \cite{hlm}}}    %
            \color{softrb}
      \pgfrect[fill]{\pgfxy(5.7,1.1)}{\pgfxy(5.5,1)}
  \color{black}
           \pgfcircle[fill]{\pgfxy(-1,-1.6)}{0.1cm}
         \pgfputat{\pgfxy(2,-1.7)}{\pgfbox[center,base]{Counterexample in general \cite{wynn2}}}      
    \end{pgfpicture}
    \caption{Weighted Weiss conjecture: Case $\beta >0$}
    \label{fig2}
\end{figure}

Due to the fact that $C$ is a $\beta$-admissible observation operator for $(T(t))_{t\ge 0}$ if and only if $C^*$ is a $(-\beta)$-admissible control operator for $(T^*(t))_{t\ge 0}$, where $\beta\in (-1,1)$, the resolvent growth conditions for $\beta$-admissible control operators can be derived from those of   $(-\beta)$-admissible observation operators.

Beside continuous-time systems we also prove a discrete-time version of the Weiss conjecture.  For $T \in \LL(\HH)$, $E \in \LL(\UU,\HH)$ and $F \in \LL(\HH, \YY)$ we
consider the  discrete time linear systems:
\beq\label{eq:state2}
   x_{n+1} =   T x_n + E u_{n+1}, \quad y_n = F x_n \qquad \hbox{with} \quad x_0 \in \HH
 \eeq
and $u_n\in \UU$, $n\in \NN$. Here, $\HH$ is the state space, $\UU$ the input space and $\YY$ is the output space of the system.
 
 Let $\beta > -1$. By $\ell^2_\beta(\UU)$ we denote the sequence space
\[ \ell^2_\beta(\UU):= \{ \{u_n\}_n \mid u_n\in \UU \mbox{ and } \|\{u_n\}_n\|^2_\beta:=\sum_{n=0}^\infty (1+n)^\beta |u_n|^2 <\infty\}. \]
Clearly, $\ell^2_\beta(\UU)$ equipped with the norm $\|\cdot\|_\beta$ is a Hilbert space.
Following \cite{hlm} and \cite{wynn10}, we say that $F$ is a 
{\em $\beta$-admissible observation operator for $T$}, if there exists a constant $M >0$ such that
 $$
        \sum_{n=0}^\infty   (1+ n)^{\beta}  \|   F T^n   x  \|^2   \le M \|x\|^2
 $$
for every $x\in \HH$. 

To test whether a given observation operator is $\beta$-admissible, a frequency-domain characterization
is convenient and, to this end, it is not difficult to show that $\beta$-admissibility of $F$ for $T$ implies the resolvent growth condition
\begin{equation}   \label{eq:resI}
         \sup_{z \in \DD}   (1-|z|^2)^{\frac{1 + \beta}{2}}   \| F ( \eins - \bar z T)^{- \beta -1}\|_{\mathcal{L}(\HH,\YY)} < \infty,
\end{equation}
where $\DD$ is the open unit disc.

The question of whether the converse
statement holds, commonly referred to as a
(weighted) Weiss conjecture, is much more subtle. For $\beta=0$, the conjecture is true if $T$ is a contraction and the output space  $\YY$ is finite-dimensional \cite{zen}.
It was shown by \cite{wynn09, wynn10} that for $T$ a normal contraction and finite-dimensional output spaces the weighted Weiss conjecture holds for positive $\beta$, but not in the case $\beta\in (-1,0)$. Moreover, the weighted Weiss conjecture holds if $T$ is a Ritt operator and a contraction for $\beta>-1$ \cite{LeMerdy}, but it is not true for general contractions if $\beta>0$, see \cite{wynn2}.
Recently, in \cite{jrw} it was shown that the Weiss conjecture holds for the forward shift on weighted Bergman spaces. One aim of this paper is to show the Weiss conjecture for adjoint operators of  $\gamma$-hypercontractions.
We obtain a characterisation of $\beta$-admissibility, $\beta>0$,  with respect to  $\gamma$-hypercontractions ($\gamma>1$)  by characterising $\beta$-admissibility with respect to the shift operator on vector-valued weighted Bergman spaces. 

It is shown in \cite{jrw} that in the case of a scalar-valued Bergman space, $\beta$-admissibility with respect to the shift operator can be characterised by the resolvent growth bound \eqref{eq:resI}. 
We extend this analysis to the vector-valued setting.

We proceed as follows. In Section 2 we introduce and study $\gamma$-hypercontractive operators and  $\gamma$-hypercontractive strongly continuous semigroups. In particular, $\gamma$-hypercontractions are unitarily equivalent to the restriction of the backward shift to an invariant subspace of a weighted Bergman space. Section 3 is devoted to the weighted Weiss conjecture for discrete-time systems. We first extend the result of  \cite{jrw} concerning the shift operator on a scalar-valued Bergman space to the vector-valued setting and then we prove that the weighted Weiss conjecture holds for $\beta>0$ if $T^*$ is a $\gamma$-hypercontraction for some $\gamma>1$. Finally, in Section 4 positive results concerning the  weighted Weiss conjecture for continuous-time systems are given.

\section{$\gamma$-hypercontractions}\label{sec2}

Let $\HH$ be a Hilbert space. For $T \in \LL(\HH)$, we define
$$
    M_T : \LL(\HH) \rightarrow \LL(\HH), \quad M_T (X) = T^* X T.
$$ 
\begin{definition}[\cite{agler85},  \cite{aem}]
Let $\HH$ be a Hilbert space and let $T \in \LL(\HH)$, $\|T\| \le 1$. Let $\gamma \ge 1$. We say that $T$
is a {\em $\gamma$-hypercontraction}, if for each $0 <r < 1$,
$$
      (\eins - M_{rT})^{\gamma}(I)  \ge 0.
$$
\end{definition}
Note that the left hand side in the definition is well-defined in the sense of the usual holomorphic functional calculus, since $\sigma(\eins-M_{rT} ) \subset \CC_+$. A $1$-hypercontraction is of course just an ordinary contraction. If $T$ is a normal contraction, then it is easy to show by the usual continuous functional calculus that $T$ is also a $\gamma$-hypercontraction for each $\gamma \ge 1$. Moreover, all strict contractions are $\gamma$-hypercontractions, as the next result shows.

\begin{theorem}
 Let $T \in \LL(\HH)$ with  $\|T\| < 1$. Then $T$ is a $\gamma$-hypercontraction for sufficiently small $\gamma>1$.
\end{theorem}

{\bf Proof:} 
Suppose that $\|T\| < 1$. Then $\|M_T\| < 1$, and $\sigma(\eins-M_T)$ is bounded away from the negative real axis,   so an analytic branch of the
logarithm exists on some open set $\Omega \supseteq \sigma(\eins-M_T)$. For $\gamma \ge 1$, define $f_\gamma(z)=\exp(\gamma \log z)$, analytic on $\Omega$.

Now $f_\gamma(z) \to z$ uniformly for $z$ in compact subsets of $\Omega$, and therefore
$f_\gamma(\eins-M_T)$, defined by the analytic functional calculus, converges to $\eins-M_T$ in the norm on $\LL(\LL(\HH))$
(see, e.g., \cite[Thm.~3.3.3]{aupetit}). 

Hence, in particular, $(\eins-M_T)^\gamma(I) \to (\eins-M_T)(I)=I-T^*T$ in norm in $\LL(\HH)$ as $\gamma \to 1$.

Since $\|T\|<1$, $\sigma((\eins-M_T)(I))$ is strictly contained in the positive real axis, and thus for sufficiently small $\gamma>1$ the spectrum of
$(\eins-M_T)^\gamma(I) $ is also strictly contained in the positive real axis, by continuity properties of the spectrum (see, e.g., \cite[Thm.~3.4.1]{aupetit}). 

Hence $(\eins-M_T)^\gamma(I) \ge 0$ for all $\gamma$ sufficiently close to $1$, and so $T$ is a $\gamma$-hypercontraction.
 \qed

If $n\in\mathbb N$, then equivalently, $T \in \LL(\HH)$ is an $n$-hypercontraction if and only if 
\[
\sum_{k=0}^{m} (-1)^k {m\choose k}T^{*k}T^k \ge 0
\]
for all $1 \le m \le n$.

In particular, a Hilbert space operator $T$ is 
{\em 2-hypercontractive\/} if
it satisfies
\[
I - T^*T \ge 0
\]
(that is, it is a contraction), and also 
\beq\label{eq:2hc}
I-2T^*T+T^{*2}T^2 \ge 0. 
\eeq
 Note, that for $1<\mu< \gamma$, the  $\gamma$-hypercontractivity property implies $\mu$-hypercontractivity. 
 
 We are particularly interested in $\gamma$-hypercontractive operators as
they are unitarily equivalent to the restriction of the backward shift to an invariant subspace of a weighted Bergman space, which we now define.

\begin{definition}\label{def:wbs}
Let $\mathbb D$ denote the open unit disk in the complex plane $\mathbb C$. For $\alpha>-1$, the 
{\em weighted Bergman space\/} $\mathcal{A}_\alpha^2(\mathbb{D},\KK), $ where $\KK$ is a Hilbert space,  contains of analytic functions $f:\mathbb{D} \rightarrow \KK$ for which 
\begin{equation}\label{eq:l2alpha}
\| f\|_{{\alpha}}^2 =  \int_\mathbb{D} \|f(z)\|^2 dA_\alpha(z) < \infty,
\end{equation}
where $dA_\alpha(z) = (1+\alpha) (1-|z|^2)^\alpha dA(z)$ and $dA(z):=\frac{1}{\pi} dx dy$ is area measure on   $\mathbb{D}$ for $z=x+iy$. 
We note that the norm $\| f \|_{{\alpha}}$ is equivalent to %
\begin{equation} \label{equivnorm}
\left( \sum_{n=0}^\infty \|f_n\|^2 (1+n)^{-(1+\alpha)} \right)^\frac{1}{2},
\end{equation}
where $f_n$ are the Taylor coefficients of $f$.
\end{definition}

For each $\alpha >-1$, let $S_\alpha$ denote the shift operator on the weighted Bergman space $A^2_{\alpha}(\DD,\KK)$,
$$
    S_\alpha f(z) = z f(z)   \quad (f \in A^2_{\alpha}(\DD,\KK) )
$$
The following theorem is a special case of Corollary 7 in \cite{aem}. For the case of integer $\gamma$, this was proved in \cite{agler85}.

\begin{theorem}   \label{thm:dilation}
Let $\alpha> -1$. Let $\HH$ be a Hilbert space and let $T \in \LL(\HH)$ be an $\alpha +2$-hypercontraction with $\sigma(T) \subset \DD$. Then $T$ is unitarily equivalent to the restriction of 
$S^*_\alpha$ to an invariant subspace of $A^2_\alpha(\DD, \KK)$, where $\KK$ is a Hilbert space.
\end{theorem}

Next we introduce the concept of $\gamma$-hypercontractive semigroups.

\begin{definition}
 Let
$(T(t))_{t \ge 0}$ be a strongly continuous contraction semigroup on a Hilbert space $\HH$, with infinitesimal generator $A$.
We call a $C_0$-semigroup $(T(t))_{t \ge 0}$  {\em $\gamma$-hypercontractive}  if each operator $T(t)$ is a $\gamma$-hyper\-contraction.
\end{definition}

In the following we assume that $(T(t))_{t \ge 0}$ is a strongly continuous contraction semigroup on a Hilbert space $\HH$, with infinitesimal generator $A$.
As in \cite{snf}, the {\em cogenerator} $T:=(A+I)(A-I)^{-1}$
exists, and is itself a contraction. Rydhe \cite{ry15} studied  the relation between $\gamma$-hypercontractivity of a strongly continuous contraction semigroup and its cogenerator.
He proved that
$T$ is $\gamma$-hypercontractive if every operator $T(t)$, $t\ge 0$, is $\gamma$-hypercontractive. Conversely,  if every operator $T(t)$, $t\ge 0$, is $N$-hypercontractive for some $N\in\mathbb N$, then $T$ is $N$-hypercontractive. However, by means of an example, Rydhe \cite{ry15} showed that for general $\gamma$-hypercontractivity this reverse implication is false. Clearly, if $A$ generates a contraction semigroup of normal operators,
then the  cogenerator of  $(T(t))_{t\ge 0}$ is $\gamma$-hypercontractive for each $\gamma\ge 1$.

In particular $2$-hypercontractivity can be characterized as follows, see  \cite{ry15}. For completeness we include a more
elementary proof, which also yields additional information.

\begin{proposition}
Let $(T(t))_{t \ge 0}$ be a strongly continuous contraction semigroup acting on a Hilbert space $\HH$. Then the following statements are equivalent.
\begin{enumerate}
\item $(T(t))_{t \ge 0}$ is 2-hypercontractive.
\item The function $t \mapsto \|T(t)x\|^2$
is convex for all $x \in H$. 
\item \beq\label{eq:nra2}
\re \langle A^2 y,y \rangle + \|Ay\|^2 \ge 0
\qquad (y \in \cD(A^2)).
\eeq
or equivalently,
\[ \|(A+A^*)x\|^2+\|Ax\|^2\ge \|A^*x\|^2\qquad (y \in \cD(A)\cap \cD(A^*)).\]
\item  The cogenerator $T$ is a 2-hypercontraction.
\end{enumerate}
\end{proposition}

\beginpf We first prove that Part 1 and  Part 2 are equivalent.
Take $t \ge 0$ and $\tau>0$. If $T(\tau)$ is a
2-hypercontraction, then, by (\ref{eq:2hc}) we have
\[
\langle T(t)x,T(t)x \rangle - 2\langle T(t+\tau)x,T(t+\tau)x\rangle
+ \langle T(t+2\tau)x,T(t+2\tau)x \rangle \ge 0,
\]
or
\beq\label{eq:convex}
\|T(t+\tau)x\|^2 \le \frac12 \left(\|T(t)x\|^2 + \|T(t+2\tau)x\|^2\right),
\eeq
which is the required convexity condition.

Conversely, the convexity condition (\ref{eq:convex})
implies that $T(\tau)$ is
a 2-hypercontraction (take $t=0$).

Next we show that Part 2 are Part 3 equivalent.
For $t>0$ and $y \in \cD(A^2)$
we calculate the second derivative of the function
$g:t \mapsto \|T(t)y\|^2$.
\[
g'(t) = \frac{d}{dt} \langle T(t)y,T(t)y \rangle
= \langle AT(t) y, T(t)y\rangle + \langle T(t) y, AT(t)y\rangle.
\]
Similarly, 
\[
g''(t) = \langle A^2 T(t)y,T(t)y \rangle
+ 2 \langle AT(t)y, AT(t)y \rangle + \langle T(t)y,A^2T(t)y \rangle.
\]
If $g$ is convex, then letting $t \to 0$ gives the condition
(\ref{eq:nra2}). 

Conversely, the condition (\ref{eq:nra2}) gives the
convexity of $t \to \|T(t)y\|^2$ for $y \in \cD(A^2)$,
and by density this holds for all $y$.

Finally we show the equivalence of Part 3 and Part 4.
We start with the condition
(\ref{eq:nra2}) and calculate
\[
\langle ( I-2T^*T+T^{*2}T^2)x,x\rangle
\]
for $x=(A-I)^2y$ (note that $(A-I)^{-2}:H \to H$ is
defined everywhere and has dense range).

We obtain
\begin{eqnarray*}
&& \langle (A-I)^2y,(A-I)^2y \rangle
- 2\langle (A^2-I)y, (A^2-I)y \rangle
+ \langle (A+I)^2y,(A+I)^2y \rangle \\
&& \qquad =
4\langle A^2 y,y \rangle + 8 \langle Ay, Ay \rangle
+ 4 \langle y, A^2y \rangle \ge 0.
\end{eqnarray*}

Thus condition (\ref{eq:nra2}) holds if and
only if the  cogenerator $T$ is 2-hyper\-contractive.
\endpf

Thus every normal contraction semigroup is  2-hypercontractive. Moreover, even every hyponormal contraction semigroup is  2-hypercontractive. Note, that a semigroup is hyponormal if the generator $A$ satisfies $D(A)\subset D(A^*)$ and $\|A^*x\|\le \| Ax\|$ for all $x\in D(A)$, see \cite{janas,OS}.
Clearly, a $C_0$-semigroup $(T(t))_{t \ge 0}$ is contractive if and only if the adjoint semigroup $(T^*(t))_{t \ge 0}$ is contractive. Unfortunately, a similar statement does not hold for 2-hypercontractions:
The right shift semigroup on $L^2(0,\infty)$ is 2-hypercontractive, but the adjoint semigroup, the left shift semigroup on $L^2(0,\infty)$ is not. 


%
%

%

\section{Discrete-time $\beta$-admissibility}

Let $\HH$, $\UU$, $\YY$ be Hilbert spaces, $T \in \LL(\HH)$, $E \in \LL(\UU,\HH)$ and $F \in \LL(\HH, \YY)$. 
Consider the  discrete time linear system:
\beq\label{eq:state1}
   x_{n+1} =   T x_n + E u_{n+1}, \quad y_n = F x_n \qquad \hbox{with} \quad x_0 \in \HH
 \eeq
and $u_n\in \UU$, $n\in \NN$. 
 
%
Following \cite{hlm} and \cite{wynn10}, we say that $F$ is a $\beta$-admissible observation operator for $T$, if there exists a constant $M >0$ such that
 $$
        \sum_{n=0}^\infty   (1+ n)^{\beta}  \|   F T^n   x  \|^2   \le M \|x\|^2
 $$
for every $x\in \HH$. Moreover, we say that $E$ is a $\beta$-admissible control operator for $T$, if there exists a constant $M>0$ such that
\[ \left\| \sum_{n=0}^\infty T^n E u_n\right\|_\HH \le M  \|\{u_n\}_n\|_\beta\]
for every  $\{u_n\}_n\in  \ell^2_\beta(\UU)$.

\begin{remark}
Let $x\in \HH$ and $\{y_n\}_n\in  \ell^2_\beta(\YY)$. Then the calculation
\begin{eqnarray*}
|\langle \{FT^n x\}_n, \{y_n\}_n\rangle_{\beta\times -\beta}|
&=& \left| \sum_{n=0}^\infty \langle FT^n x, y_n\rangle_\YY\right|\\
&=& \langle x, \sum_{n=0}^\infty (T^*)^n F^* y_n \rangle_\HH
\end{eqnarray*}
implies that $F$ is a $\beta$-admissible observation operator for $T$ if and only if $F^*$ is a  $(-\beta)$-admissible control operator for $T^*$.
\end{remark}

A characterisation of $\beta$-admissibility with respect to  $\gamma$-hypercontractions ($\gamma>1$) may be obtained by characterising $\beta$-admissibility with respect to the shift operator on vector-valued weighted Bergman spaces, as defined in Definition~\ref{def:wbs}.

It is shown in \cite{jrw} that in the case of a scalar-valued Bergman spaces, $\beta$-admissibility with respect to $S_\alpha$ can be characterised by the resolvent growth bound \eqref{eq:resI}. This result was obtained by noting that $\beta$-admissibility is equivalent to boundedness of an appropriate little Hankel operator, while \eqref{eq:resI} is equivalent to boundedness of the same Hankel operator on a set of reproducing kernels. That such Hankel operators satisfy a Reproducing Kernel Thesis (boundness on the reproducing kernels is equivalent to operator boundedness) is equivalent to the characterisation of $\beta$-admissibility by the growth condition \eqref{eq:resI}. \\

To extend this analysis to the  vector-valued setting, let $\KK,\YY$ be Hilbert spaces and consider an analytic function   $C:\mathbb{D} \rightarrow \mathcal{L}(\KK,\YY)$  given by 
\[
C(z) = \sum_{n=0}^\infty C_n z^n, \qquad z \in \mathbb{D},
\]
where $C_n \in \mathcal{L}(\KK,\YY)$, for each $n$.  
We write $L^2_\alpha(\DD,\KK)$ for the space of measurable functions $f: \DD \to \KK$
satisfying \eqref{eq:l2alpha}.
We also write 
\[
\overline{A^2_\alpha}( \DD,\KK) = \{ z \mapsto  g(\overline z) : g \in A^2_\alpha(\DD,\KK) \}.
\]
The little Hankel operator $h_C :A_{\beta-1}^2(\DD,\KK) \rightarrow \overline{A^2_\alpha}( \DD, \YY)$ acting between weighted Bergman spaces is defined by 
\begin{equation} \label{eq:lh}
h_C(f) := \overline{P_\alpha}( C( \overline{\iota} ) f(\iota)), \qquad f \in A_{\beta-1}^2(\DD,\KK),
\end{equation}
where $\overline{P_\alpha}: L^2_\alpha(\DD,\KK) \rightarrow \overline{A^2_\alpha}( \DD,\KK) $ is the orthogonal projection onto the anti-analytic functions and $\iota(z)=z, z \in \mathbb{D}$. The following result links $\beta$-admissibility with little Hankel operators of the form \eqref{eq:lh}. 

\begin{proposition} \label{prop:rkt}
Let $\alpha>-1$ and $\beta>0$. Let $\KK$, $\YY$ be Hilbert spaces. Given $F \in \mathcal{L}(A^2_\alpha(\DD,\KK),\YY)$, define bounded linear operators $F_n \in \mathcal{L}(\KK,\YY)$ by
\[
F_n x = F( x \iota^n), \qquad x \in \KK, n \in \mathbb{N},
\]
and symbols $C:\mathbb{D} \rightarrow \mathcal{L}(\KK,\YY)$, $\tilde C: \DD \rightarrow \mathcal{L}(\YY,\KK)$ by 
\[
C(z) = \sum_{n=0}^\infty (1+n)^\alpha F_n z^n, \qquad \tilde C(z)  = \sum_{n=0}^\infty (1+n)^\alpha F_n^\ast z^n.
\]
The following conditions are equivalent:
\begin{enumerate}
\item[(i)] The resolvent condition \eqref{eq:resI} holds with $T=S_\alpha$ and $\HH = A^2_\alpha(\DD,\KK)$;
\item[(ii)] The Hankel operator $h_{\tilde C}:A_{\beta-1}^2(\DD,\YY) \rightarrow \overline{A^2_\alpha}( \DD, \KK)$ satisfies 
\[
\sup_{\omega \in \mathbb{D}, \|y\|_\YY = 1} \| h_{\tilde C} k_{\omega,y}^{\beta-1} \|_{\overline{A^2_\alpha}(\DD,\KK)} < \infty,
\]
where
\[
k_{\omega,y}^{\beta-1} (z) := y \frac{(1-|\omega|^2)^{\frac{1+\beta}{2}} }{(1- \bar \omega z )^{1+\beta}}, \qquad z,\omega \in \mathbb{D}, \; y \in \YY,
\]
are the normalized reproducing kernels for $A_{\beta-1}^2(\DD,\YY)$;
\item[(iii)] The Hankel operator $h_C: A_{\beta-1}^2(\DD,\KK) \rightarrow \overline{A^2_\alpha}( \DD, \YY)$ satisfies 
\[
h_C \in \mathcal{L}( A^2_{\beta-1}(\DD,\KK), \overline{A^2_\alpha}(\DD,\YY) ); 
\]
\item[(iv)] $F$ is $\beta$-admissible for $S_\alpha$ on $A^2_\alpha(\DD,\KK)$. 
\end{enumerate}
\end{proposition}
\proof
$(i) \Leftrightarrow (ii)$ follows directly from a vectorial analogue of \cite[Proposition 2.3 (ii)]{jrw}. \\

$(ii) \Rightarrow (iii)$ Note first that \cite[Theorem 2.7]{jrw} extends to the vector-valued setting to imply that $h_{\tilde C}: A_{\beta-1}^2(\DD,\YY) \rightarrow \overline{A^2_\alpha}( \DD, \KK)$ is bounded. An alternative characterisation of boundedness little Hankel operators can be given in terms of generalized Hankel matrices of the form 
\[
\Gamma_\Phi^{a, b} := \left( (1+m)^a (1+n)^b \Phi_{n+m} \right)_{m,n \geq 0}
\]
where $a,b>0$ and $\Phi:\mathbb{D} \rightarrow \mathcal{L}(\HH_1,\HH_2)$ is given by $\Phi(z) = \sum_{n \geq 0} \Phi_n z^n$, for some Hilbert spaces $\HH_1,\HH_2$. In particular,  the vectorial analogue of \cite[Proposition 2.3 (i)]{jrw} implies that
\begin{equation} \label{eq:lh_gh}
h_{\tilde C} \in \mathcal{L}\left( A_{\beta-1}^2(\DD,\YY) , \overline{A^2_\alpha}( \DD, \KK) \right) \Longleftrightarrow \Gamma_{\tilde C}^{\frac{\beta}{2},\frac{1+\alpha}{2}} \in \mathcal{L}( \ell^2(\YY),\ell^2(\KK)). 
\end{equation}
Now, it is shown in \cite[Theorem 9.1]{peller} that 
\begin{equation} \label{eq:gh_charac}
\Gamma_{\tilde C}^{\frac{\beta}{2},\frac{1+\alpha}{2}} \in \mathcal{L}( \ell^2(\YY),\ell^2(\KK)) \Leftrightarrow \tilde C \in \Lambda_{\frac{1+\alpha+\beta}{2}}\left( \mathcal{L}(\YY,\KK) \right).
\end{equation}
Here, for $s>0$ and a Banach space $X$, $\Lambda_s(X)$ is the Besov space containing functions $f \in L^\infty(\DD,X)$ for which
\[
\sup_{\tau \in \TT, \tau \neq 1} \frac{ \|\Delta_\tau^n f\|_{L^\infty(\DD,X)}}{|1-\tau|^s} <\infty, \qquad \bigg( (\Delta_\tau f)(\xi) := f(\xi \tau) - f(\tau), \; \Delta_\tau^n :=\Delta_\tau \Delta_{\tau}^{n-1} \bigg), 
\]
for some integer $n>s$. It follows immediately that $C \in \Lambda_{\frac{1+\alpha+\beta}{2}}\left( \mathcal{L}(\KK,\YY) \right)$ and hence, by \eqref{eq:lh_gh} and \eqref{eq:gh_charac}, that 
\[
h_C \in \mathcal{L}( A^2_{\beta-1}(\DD,\KK), \overline{A^2_\alpha}(\DD,\YY) ).
\]

$(iii) \Leftrightarrow (iv)$: The vectorial analogue of \cite[Proposition 2.1]{jrw} implies that $(iv)$ holds if and only if $\Gamma_C^{\frac{1+\alpha}{2},\frac{\beta}{2}} \in \mathcal{L}( \ell^2(\KK),\ell^2(\YY))$. By \eqref{eq:lh_gh}, boundedness $(iii)$ of the little Hankel operator $h_C$ is equivalent to $\Gamma_C^{\frac{\beta}{2},\frac{1+\alpha}{2}} \in \mathcal{L}( \ell^2(\KK),\ell^2(\YY))$. That $(iii)$ and $(iv)$ are equivalent then follows from \cite[Theorem 9.1]{peller} and the fact that $\alpha>-1,\beta>0$. \\

$(iv) \Rightarrow (i)$ is well known. See, for example, \cite{wynn2}. 
\qed

\begin{theorem}\label{thm:discreteadm}
Let $\beta > 0$. Let $\HH$, $\YY$ be Hilbert spaces and let $T^* \in \LL(\HH)$ be a $\gamma$-hypercontraction for some $\gamma >1$. 
Let $F \in \LL(\HH, \YY)$. Then the following are equivalent:
\begin{enumerate}
\item $F$ is a  $\beta$-admissible observation operator for $T$.
\item
$$
      \sup_{z \in \DD}   (1-|z|^2)^{\frac{1 + \beta}{2}}   \| F ( \eins - \bar z T)^{- \beta -1}\|_{\LL(\HH, \YY)}  < \infty.
$$
\end{enumerate}
\end{theorem}
\proof
The implication (1) $\Rightarrow$ (2) follows as usual from the testing on fractional derivatives of reproducing kernels. For (2) $\Rightarrow$ (1), 
write   $K=    \sup_{z \in \DD}   (1-|z|^2)^{\frac{1 + \beta}{2}}   \| F ( \eins - \bar z T)^{- \beta -1}\|_{\LL(\HH, \YY)}$ and      let us first replace $T$ by $rT$ for some
$0 < r < 1$. Write $\gamma = 2 + \alpha$. By Theorem \ref{thm:dilation}, $(rT)^*$ is the restriction of $S_\alpha^*$ to the invariant subspace $\HH \subset A^2_\alpha(\DD, \KK)$. 
Extend $F$ trivially to $A^2_\alpha(\DD, \KK)$ by letting $F=0$ on $\HH^\perp \subset A^2_\alpha(\DD, \KK)$. Then $F^* y \in \HH$ for all $y \in \YY$.
Then for each $z \in \DD$ we obtain
\begin{eqnarray}
 \| F ( \eins - \bar z S_\alpha)^{- \beta -1}\|_{\LL(A_\alpha^2(\DD,\KK), \YY)} &=& \sup_{h \in A^2_\alpha(\KK), \|h\|=1}    \| F ( \eins - \bar z S_\alpha)^{- \beta -1} h\|_\YY \nonumber\\
 &= & \sup_{h \in A^2_\alpha(\KK) , \|h\|=1}   \sup_{y \in \YY, \| y\|=1}  | \langle ( \eins - \bar z S_\alpha)^{- \beta -1} h, F^* y  \rangle    |  \nonumber \\
&= & \sup_{h \in A^2_\alpha, \|h\|=1}   \sup_{y \in \YY, \| y\|=1} | \langle  h,       ( \eins -  z S_\alpha^*)^{- \beta -1}   F^* y \rangle    |   \nonumber\\
&= & \sup_{h \in A^2_\alpha, \|h\|=1}   \sup_{y \in \YY, \| y\|=1} | \langle  h,       ( \eins -  z (rT)^*)^{- \beta -1}              F^* y \rangle    |  \nonumber \\
&= & \sup_{h \in \HH, \|h\|=1}   \sup_{y \in \YY, \| y\|=1} | \langle  h,       ( \eins -  z (rT)^*)^{- \beta -1}              F^* y \rangle    |   \nonumber\\
&=& \| F ( \eins - \bar z r T )^{- \beta -1}\|_{\LL(\HH, \YY)}  \nonumber \\
&\le & K \frac{1}{ (1-|rz|^2)^{\frac{1 + \beta}{2}}  } \nonumber \\
&\le &K \frac{1}{ (1-|z|^2)^{\frac{1 + \beta}{2}}  }. \label{eq:F_res_bnd}
\end{eqnarray}
Hence, by Proposition \ref{prop:rkt},  $F$ is an  $\beta$-admissible observation operator for $S_\alpha$. \\

Thus there exists a constant $M$ such that for each $x \in \HH$,
\begin{eqnarray*}
     \sum_{n=0}^\infty  (1+ n)^{\beta}      \| F (rT)^n x \|^2_\YY &= &   \sum_{n=0}^\infty  (1+ n)^{\beta}   \sup_{y \in \YY, \| y\|=1}  | \langle  (rT)^n x , F^* y    \rangle |_\YY^2  \\
&= &   \sum_{n=0}^\infty   (1+ n)^{\beta}  \sup_{y \in \YY, \| y\|=1}  | \langle   x ,   ( (rT)^n)^*  F^* y    \rangle  |^2   \\
&= &   \sum_{n=0}^\infty  (1+ n)^{\beta}  \sup_{y \in \YY, \| y\|=1}  | \langle   x ,   ( S_\alpha^n) ^*  F^* y    \rangle  |^2   \\
&= &   \sum_{n=0}^\infty  (1+ n)^{\beta}  \sup_{y \in \YY, \| y\|=1}  | \langle    S_\alpha^n   x ,    F^* y    \rangle  |^2   \\
&= &   \sum_{n=0}^\infty   (1+ n)^{\beta}  \|   F S_\alpha^n   x  \|^2_\YY   \le M \|x\|^2 \\
\end{eqnarray*}
Here, the constant $M$ depends only on $K$, $\alpha$ and $\beta$, but not on $r$. It therefore follows easily from the Monotone Convergence Theorem that
$$
 \sum_{n=0}^\infty  (1+ n)^{\beta}      \| F T^n x \|_\YY^2 \le M \|x\|^2   \qquad (x \in \HH)
$$
and $F$ is a $\beta$-admissible observation operator for $T$. 
\qed

By duality we obtain the following result.

\begin{theorem}
Let $\beta \in (-1,0)$. Let $\HH$, $\UU$ be Hilbert spaces and let $T \in \LL(\HH)$ be a $\gamma$-hypercontraction for some $\gamma >1$. 
Let $E \in \LL(\UU, \HH)$. Then the following are equivalent:
\begin{enumerate}
\item $E$ is a  $\beta$-admissible control operator for $T$.
\item
$$
      \sup_{z \in \DD}   (1-|z|^2)^{\frac{1 + \beta}{2}}   \|  ( \eins - \bar z T)^{- \beta -1}E\|_{\LL(\HH, \YY)}  < \infty.
$$
\end{enumerate}
\end{theorem}

\begin{remark}
Theorem \ref{thm:discreteadm} in particular shows  Wynn's result \cite{wynn10} for $\beta$-admissibility of normal discrete contractive semigroups, also for infinite-dimensional 
output space.
\end{remark}

\section{Continuous-time $\beta$-admissibility}

We consider a continuous-time control system of the form
\begin{eqnarray*}
\dot{x}(t) &=& A x(t) + B u(t), \quad x(0)=x_0, \, t\ge 0,\\
y(t)    &=& Cx(t), \, t\ge 0.
\end{eqnarray*}
Here $A$ is the generator of a $C_0$-semigroup   $(T(t))_{t \ge 0}$ on a Hilbert space $\HH$.
Writing $\HH_1=D(A)$ and $\HH_{-1}=D(A^*)^*$, we suppose that
 $B\in L(\UU,\HH_{-1})$ and $C\in L(\HH_1, \YY)$, where $\UU$ and $\YY$ are Hilbert spaces as well. 

\begin{definition}
 Let $\beta>-1$.
\begin{enumerate}
\item $B$ is called a {\em $\beta$-admissible control operator for $(T(t))_{t \ge 0}$}, if
there exists a constant $M>0$ such that
\[   \left\| \int_0^\infty T(t) B u(t) \, dt \right \| \le M \|u\|_{L^2_\beta (0,\infty; \UU)}  \]
for every $u \in L^2_\beta (0,\infty; \UU)$.
\item $C$ is called a {\em $\beta$-admissible observation operator for $(T(t))_{t \ge 0}$}, if
there exists a constant $M>0$ such that
\[    \int_0^\infty  t^\beta\|CT(t) x\|^2 \, dt  \le M \|x\|_{\HH}^2  \]
for every $x\in \HH_1$.
\end{enumerate}
\end{definition}

\begin{remark}\label{rem:duality}{\rm 
Similarly as for discrete-time systems it can be shown  for $\beta\in (-1,1)$  that $B$ is a $\beta$-admissible control operator for $(T(t))_{t \ge 0}$ if and only if $B^*$ is a $(-\beta)$-admissible observation operator for $(T^*(t))_{t \ge 0}$.}
\end{remark}

The following result is proven in \cite[Propositions 2.1 and 2.2]{wynn2} for $\beta \in (0,1)$. The trivial extension to the case $\beta>0$ is given for completeness. 
For $\alpha>-1$ we write $A^2_\alpha(\CC_+)$ for the Bergman space on the right half-plane corresponding to the measure $x^\alpha \, dx \, dy$.

\begin{proposition}\label{propequiv}
Let $\beta>0$. Suppose that $A$ generates a contraction semigroup on $\HH$ and that $C\in \LL(D(A),\YY)$. Define the cogenerator $T\in \LL(\HH)$ by $T:=(I+A)(I-A)^{-1} $ and 
$F:= C(I-A)^{-(1+\beta)}\in \LL(\HH,\YY)$. Then the following statements hold.
\begin{enumerate}
\item $C$ is a $\beta$-admissible observation operator for  $(T(t))_{t \ge 0}$ if and only if $F$ is a $\beta$-admissible observation operator for $T$.
\item The resolvent condition \eqref{eq:resI} for $(F,T)$ holds if and only if 
\[ \sup_{\lambda\in \mathbb C_+} (\mbox{Re}\,\lambda)^{\frac{1+\beta}{2}} \|C(\lambda-A)^{-(1+\beta)}\|<\infty.\]
\end{enumerate}
\end{proposition}
\proof
$1.$ $F$ is $\beta$-admissible for $T$ if and only if $\Lambda : A_{\beta-1}^2(\DD) \rightarrow \LL(\HH,\YY)$  defined initially on reproducing kernels by $\Lambda f = Ff(T)$ extends to a bounded linear operator. On the other hand, $C$ is $\beta$-admissible for $A$ if and only  if $\tilde \Lambda: A^2_{\beta-1}(\mathbb{C}_+) \rightarrow \LL(\HH,\YY)$ defined initially on reproducing kernels by $\tilde \Lambda (g) = Cg(-A)$ extends to a bounded linear operator. That the two conditions are equivalent follows from the fact that for any $\beta>0$ there is an  isomorphism $J_\beta : A_{\beta-1}^2(\DD) \rightarrow A^2_{\beta-1}(\mathbb{C}_+)$ for which $\Lambda = \tilde \Lambda \circ J_\beta$ holds on each reproducing kernel. \\

$2.$ Follows directly from the identities
\[
D(I-\bar zT)^{-(1+\beta)} = \frac{C R\left( \frac{1-\bar z}{1+\bar z},A\right)^{1+\beta}}{(1+\bar z)^{1+\beta}}, \qquad z \in \mathbb{D}
\]
and
\[
\text{Re}\left( \frac{1-z}{1+z} \right)|1+z|^2 = (1-|z|^2), \qquad z \in \mathbb{D}.
\]
\qed

Our main theorems concerning continuous-time systems are as follows.

\begin{theorem}\label{thm:contadm}
Let $\beta > 0$. Let $(T(t))_{t\ge 0}$ be a contraction semigroup on $\HH$ such that the adjoint of  the cogenerator  $T^*$ is $\gamma$-hypercontractive for some $\gamma >1$. Then the following are equivalent:
\begin{enumerate}
\item $C$ is $\beta$-admissible observation operator for $(T(t))_{t\ge 0}$.
\item
$$
      \sup_{\lambda\in \mathbb C_+} (\mbox{Re}\,\lambda)^{\frac{1+\beta}{2}} \|C(\lambda-A)^{-(1+\beta)}\|<\infty.
$$
\end{enumerate}
\end{theorem}

\begin{proof} The statement of the theorem follows from Proposition \ref{propequiv} together with Theorem \ref{thm:discreteadm}. \qed
\end{proof}

\begin{remark}
$T^*$ is $\gamma$-hypercontractive if every operator $T^*(t)$, $t\ge 0$, is $\gamma$-hypercontractive.  If $A$ generates a contraction semigroup of normal operators,
then the adjoint of the cogenerator of  $(T(t))_{t\ge 0}$ is $\gamma$-hypercontractive for each $\gamma\ge 1$, see Section \ref{sec2}.
\end{remark}

By duality we obtain the following result.

\begin{theorem}\label{thm:contadm2}
Let $\beta \in (-1,0)$. Let $(T(t))_{t\ge 0}$ be a contraction semigroup on $\HH$ such that the cogenerator $T$ is $\gamma$-hypercontractive for some $\gamma>1$. Then the following are equivalent:
\begin{enumerate}
\item $B$ is $\beta$-admissible control operator for $(T(t))_{t\ge 0}$.
\item
$$
      \sup_{\lambda\in \mathbb C_+} (\mbox{Re}\,\lambda)^{\frac{1+\beta}{2}} \|(\lambda-A)^{-(1+\beta)}B\|<\infty.
$$
\end{enumerate}
\end{theorem}

Theorems \ref{thm:contadm} and \ref{thm:contadm2} give positive results  for $\beta>0$ and adjoints of $\gamma$-hypercontractions in the case of observation operators, and for $\beta<0$ and $\gamma$-hypercontractions
in the case of control operators. The remaining possibilities for $\beta \in (-1,0) \cup (0,1)$ can be shown not to hold by means of various counterexamples.
For $\beta \in (-1,0)$ the counterexample for normal semigroups given in \cite{wynn09} shows that there is no positive result for observation operators in either the $\gamma$-hypercontractive or
adjoint $\gamma$-hypercontract\-ive case. For $\beta \in (0,1)$, there is a counterexample in  \cite{wynn09} based on the unilateral shift, which is 2-hypercontractive, see Figure~\ref{fig3}.

By Remark \ref{rem:duality}, these provide appropriate counterexamples for control operators as well.

 \begin{figure}[h]
 \centering
 \begin{pgfpicture}{0cm}{2cm}{10cm}{5.5cm}
   
        \pgfsetlinewidth{1pt}
        \pgfxyline(12.7,2)(12.7,5)
        \pgfxyline(-2.7,5)(12.7,5)
         \pgfxyline(6,2)(6,5)
          \pgfxyline(-2.7,2)(12.7,2)
              \pgfxyline(-2.7,3)(12.7,3)
        \pgfxyline(-0.7,2)(-0.7,5)
          \pgfxyline(-2.7,2)(-2.7,5)
      \pgfsetlinewidth{0.8pt}
            \pgfxyline(-2.7,4)(12.7,4)
    \pgfputat{\pgfxy(2.5,4.3)}{\pgfbox[center,base]{$T$ $\gamma$-hypercontr.~for some $\gamma>1$}}
     \pgfputat{\pgfxy(9.5,4.3)}{\pgfbox[center,base]{$T^*$ $\gamma$-hypercontr.~for some $\gamma>1$}}
      \pgfputat{\pgfxy(-1.7,3.4)}{\pgfbox[center,base]{$\beta\in(-1,0)$}}
     \pgfputat{\pgfxy(9.6,2.3)}{\pgfbox[center,base]{Conjecture holds by Theorem \ref{thm:contadm} }}
       \pgfputat{\pgfxy(9.6,3.4)}{\pgfbox[center,base]{Counterexample \cite{wynn09}}}
     \pgfputat{\pgfxy(2.7,2.3)}{\pgfbox[center,base]{Counterexample \cite{wynn09}}}
      \pgfputat{\pgfxy(2.7,3.4)}{\pgfbox[center,base]{Counterexample \cite{wynn09}}}
     \pgfputat{\pgfxy(-1.7,2.3)}{\pgfbox[center,base]{$\beta\in(0,1)$}}
    \end{pgfpicture}
    \caption{Weighted Weiss conjecture}
    \label{fig3}
\end{figure}

%
%
%
%
%
%
%
%
%
%
%
%
%

\end{document}